\documentclass[11pt,reqno]{amsart}

\usepackage{a4wide} 
\usepackage{lineno}
\modulolinenumbers[5]
\usepackage{amssymb,euscript,enumerate}
\usepackage{color}
\usepackage{graphicx}
\usepackage{epstopdf}
\usepackage[colorlinks=true,citecolor=red,linkcolor=blue]{hyperref}
\usepackage[]{cleveref}
\usepackage{mathtools}
\usepackage{esint}
\usepackage{mathrsfs}
\usepackage[notref, notcite, final]{showkeys}

\usepackage{soul}			

\theoremstyle{plain}
\newtheorem{theorem}{Theorem}[section]
\newtheorem{lemma}[theorem]{Lemma}

\newtheorem{proposition}[theorem]{Proposition}

\newtheorem{thmx}{Theorem}

\theoremstyle{definition}
\newtheorem{definition}[theorem]{Definition}

\theoremstyle{remark}
\newtheorem{remark}[theorem]{Remark}

\numberwithin{equation}{section}

\newcommand{\R}{{\mathbb R}}

\newcommand{\bS}{\mathbb{S}}

\newcommand{\cL}{{\mathcal L}}

\newcommand{\al}{\alpha}

\newcommand{\ga}{\gamma}
\newcommand{\de}{\delta}
\newcommand{\e}{\varepsilon}

\newcommand{\la}{\lambda}
\newcommand{\si}{\sigma}
\newcommand{\vp}{\varphi}

\newcommand{\Si}{\Sigma}
\newcommand{\Om}{\Omega}

\newcommand{\ti}{\times}

\newcommand{\pa}{\partial}
\newcommand{\su}{\subset}
\newcommand{\qu}{\quad}

\newcommand{\sm}{\setminus}
\newcommand{\ra}{\rightarrow}

\newcommand{\D}{\nabla}

\newcommand{\diam}{\operatorname{diam}}

\newcommand{\fr}{\frac}

\newcommand{\inn}[2]{\left\langle {#1},{#2} \right\rangle}

\def\({\left(}
\def\){\right)}
\def\<{\left\langle}
\def\>{\right\rangle}

\makeatletter
\@namedef{subjclassname@2020}{\textup{2020} Mathematics Subject Classification}
\makeatother

\title[Curvature bound for $L_p$ Minkowski problem]{Curvature bound for $L_p$ Minkowski problem}

\author{Kyeongsu Choi}
\address{School of Mathematics, Korea Institute for Advanced Study, 85 Hoegiro, Dongdaemun-gu, Seoul 02455, Republic of Korea}
\email{choiks@kias.re.kr}

\author{Minhyun Kim}
\address{Department of Mathematics \& Research Institute for Natural Sciences, Hanyang University, 04763 Seoul, Republic of Korea}
\email{minhyun@hanyang.ac.kr}

\author{Taehun Lee}
\address{School of Mathematics, Korea Institute for Advanced Study, 85 Hoegiro, Dongdaemun-gu, Seoul 02455, Republic of Korea}
\email{taehun@kias.re.kr}


\begin{document}

\begin{abstract}	
We establish curvature estimates for anisotropic Gauss curvature flows. By using this, we show that given a measure $\mu$ with a positive smooth density $f$, any solution to the $L_p$ Minkowski problem in $\mathbb{R}^{n+1}$ with $p \le -n+2$ is a hypersurface of class $C^{1,1}$. This is a sharp result because for each $p\in [-n+2,1)$ there exists a convex hypersurface of class $C^{1,\frac{1}{n+p-1}}$ which is a solution to the $L_p$ Minkowski problem for a positive smooth density $f$. In particular, the $C^{1,1}$ regularity is optimal in the case $p=-n+2$ which includes the logarithmic Minkowski problem in $\mathbb{R}^3$.
\end{abstract}	

\maketitle
\sffamily

\section{Introduction}

The Minkowski problem asks if there exists a convex body whose surface area measure is a prescribed measure on $\bS^n$. This problem was first proposed and solved by Minkowski \cite{Min03} for special cases, and later extended independently by Aleksandrov \cite{Ale38} and Fenchel--Jessen \cite{FJ38} to general measures. The regularity of solutions to the Minkowski problem has been established by Pogorelov \cite{Pog73,Pog78}, Cheng--Yau \cite{CY76}, and Caffarelli \cite{Caf89,Caf90b,Caf90a,Caf91} from the viewpoint of differential geometry and partial differential equations.

\bigskip

The study of the $L_{p}$ Minkowski problem was initiated by Lutwak \cite{Lut93} as an important variant of the Minkowski problem. The $L_{p}$ Minkowski problem asks whether a given measure $\mu$ on $\mathbb{S}^{n}$ arises as the $L_{p}$ surface area measure $S_{p}$ of a convex body $\Omega$, i.e.,
\begin{equation} \label{eq:MP}
S_{p}(\Omega, \cdot) = \mu,
\end{equation}
see \Cref{sec:MP} for details. Besides the classical Minkowski problem ($p=1$), there are some significant cases such as the logarithmic Minkowski problem ($p=0$) that characterizes the cone volume, and the centro-affine Minkowski problem ($p=-n-1$) that is invariant under affine transformations. The existence of solutions have been extensively studied over the last decades in various contexts, see \cite{HLYZ05,Lut93,LYZ04} for $p>1$, \cite{BBC20,BBCY19,BT17,HLYZ10,KL24,Zhu15b} for $p \in (0,1)$, \cite{BHZ16,BLYZ13,CL18,CLZ19,Zhu14} for $p=0$, and \cite{BBCY19,GLW22,LW13,Zhu15a,Zhu17} for $p<0$. Most of the aforementioned references also concern the uniqueness problem. We refer the readers to \cite{BLYZ12,BCD17,CFL22,CHLL20,Ch17,Sta02,Sta03,XL16} for further uniqueness results and \cite{CLZ17,CLZ19,HLW16,JLW15,Yag06} for non-uniqueness results.

\bigskip

In this work, we focus on the regularity of solutions when the prescribed measure on $\mathbb{S}^{n}$ has a density $f$, namely $\mathrm{d}\mu = f \,\mathrm{d} \si$, where $\si$ is the spherical Lebesgue measure. In this case, the $L_{p}$ Minkowski problem can be interpreted as a Monge--Amp\'ere type equation
\begin{align} \label{eq:MA}
\det(u_{ij} + u\delta_{ij}) = fu^{p-1} \qu \text{on } \bS^n
\end{align}
in terms of the support function $u$ of a convex body, provided that $u$ is positive. It is worth noting that the support function of a convex body is positive if and only if the origin lies in the interior of a convex body, in which case the standard regularity theory by Caffarelli \cite{Caf90b} can be applied to \eqref{eq:MA}. Therefore, in order to study the regularity, we need to focus on cases where the origin lies on the boundary, that is, when the support function of a convex body is nonnegative.

If $u$ is nonnegative, the equation \eqref{eq:MA} is understood in the generalized sense (of Aleksandrov). We say that the support function $u \in C(\mathbb{S}^{n})$ of a convex body $\Omega$ or its convex hypersurface $\Si=\pa \Om$ is called a {\it generalized solution of \eqref{eq:MA}} if $\Omega$ satisfies \eqref{eq:MP}, namely,
\begin{equation} \label{eq:gen-sol}
\int_{E} \mathrm{d}S_{p}(\Omega, z) = \int_{E} f \,\mathrm{d}\sigma \quad\text{for all Borel sets } E \subset \mathbb{S}^{n},
\end{equation}
where $\sigma$ is the spherical Lebesgue measure.

\bigskip

Let us briefly review the regularity results for the $L_p$ Minkowski problem in the literature. When $p \in (-n-1, -n+1] \cup [n+1, \infty)$, Chou and Wang proved \cite{CW06} that solutions to \eqref{eq:MA} have the origin in their interior, and therefore their support functions are of class $C^{2,\gamma}(\mathbb{S}^{n})$, provided that $f \in C^{\gamma}(\mathbb{S}^{n})$ is positive. In particular, solutions to \eqref{eq:MA} are smooth if $f$ is smooth and positive. For the case $p\in (-\infty,  -n-1]$, the smoothness of the solutions can be deduced from the result of Andrews \cite{And00}. We note that, for the affine critical case where $p=-n-1$ (in which the equation is invariant under affine transformations), solutions might not exist in general, as described in \cite{CW06}.

However, solutions are not necessarily smooth when $p \in (-n+1,1) \cup (1, n+1)$. As indicated above, if the origin lies on the boundary of convex bodies, then one cannot apply the standard regularity theory by Caffarelli, which may allow non-smooth solutions. Indeed, the $L_p$ Minkowski problem, even with sufficiently regular $f$ could have a solution with the origin on its boundary.
See \cite[Example 1.6]{BT17} for $p \in (-n+1,1)$ and \cite[Example 4.1]{HLYZ05} for $p \in (1, n+1)$. In these examples, the hypersurfaces are of class $C^{k,\ga}$ with $\ga\in (0,1]$ for $k+\ga=\fr{2n}{n+p-1}$. Here,
\begin{center}
    \textit{a complete hypersurface $\Sigma$ embedded in $\mathbb{R}^{n+1}$ is said to be of class $C^{k,\gamma}$,}
\end{center}
  if each $p\in \Sigma$ has a neighborhood $U$ such that $\Sigma \cap U$ is a rotated graph of a $C^{k,\gamma}$-function.
See also \cite[Example 4.2]{BBC20} for examples of convex bodies with a flat side although the associated $f$ is not smooth. 

We summarize the aforementioned known results in terms of $C^{1,1}$ regularity of hypersurfaces, which is equivalent to the boundedness of principal curvatures. For $p \in (-\infty, -n+1] \cup \{1\} \cup [n+1, \infty)$, solutions are smooth if $f$ is smooth. However, for $p \in (1, n+1)$, there exist solutions that are not $C^{1,1}$ even though $f$ is smooth. For the remaining range, $p \in (-n+1, 1)$, it was not known whether the solutions are $C^{1,1}$ or not when $f$ is smooth.

\bigskip

In this paper, we completely investigate the $C^{1,1}$ regularity for the $L_p$ Minkowski problem with $p\in (-n+1,1)$ when $f$ is smooth. More precisely, we show that any generalized solution to \eqref{eq:MA} is a hypersurface of class $C^{1,1}$ for $p\in (-n+1,-n+2]$; however, there exist solutions that are not $C^{1,1}$ when $p\in(-n+2,1)$.

\begin{theorem} \label{thm:main-MP}
Given $p\in(-\infty, -n+2]$ and a positive and smooth function $f$, any generalized solution $\Sigma$ to \eqref{eq:MA} is a hypersurface of class $C^{1,1}$. Moreover, its principal curvatures are uniformly bounded by a constant depending only on $n$, $p$, the diameter of $\Sigma$, $\|f\|_{C^2(\mathbb{S}^n)}$, and $\min_{\mathbb{S}^{n}}f$.
\end{theorem}

This is a sharp result due to the following theorem.

\begin{theorem}\label{thm:counter-MP}
Let $p\in [-n+2,1)\cup (1,n+1)$. 
\begin{enumerate}[\normalfont(i)]
\item If $p \in [-n+2, 1)$, then there exists a generalized solution $\Sigma$ to \eqref{eq:MA} such that $\Sigma$ is a hypersurface of at most class $C^{1,\frac{1}{n+p-1}}$ and $f$ is a positive smooth function.

\item If $p\in (1,n+1)$, then there exists a generalized solution $\Sigma$ to \eqref{eq:MA} such that $\Sigma$ is a hypersurface of at most class $C^{1,\frac{n-p+1}{n+p-1}}$ and $f$ is a positive smooth function.
\end{enumerate}
\end{theorem}

\begin{remark}
\begin{enumerate}[(i)]
\item By combining the known results with \Cref{thm:main-MP} and \Cref{thm:counter-MP}, we conclude the following: the $L_p$ Minkowski problem enjoys $C^{1,1}$ regularity of hypersurfaces if and only if $p\in (-\infty,-n+2]\cup \{1\}\cup [n+1,\infty)$.

\item The $C^{1,1}$ regularity is optimal for the case of $p=-n+2$, which includes the logarithmic Minkowski problem in $\R^3$ ($n=2$, $p=0$) as a special case.
\item 
Chen--Feng--Liu \cite{CFL22}  established  a uniform diameter estimate for the logarithmic Minkowski problem in $\R^3$. Thus, the curvature bound in \Cref{thm:main-MP} depends only on $\|f\|_{C^2(\mathbb{S}^n)}$ and $\min_{\mathbb{S}^n}f$ in this case.

On the other hand, in the case $-n-1<p<0$,  there is no uniform diameter estimate depending on $f$ by the result in Jian--Lu--Wang \cite{JLW15}.
\item 
Although the smoothness of solutions and curvature estimates are already known for $p \in (-\infty, -n+1]$ in Andrews \cite{And00} and Chou--Wang \cite{CW06}, we included this range of $p$ in \Cref{thm:main-MP} because our curvature estimates also work in this range and the proofs are new.

\item It was previously known in Chou--Wang \cite{CW06} that when $p \in (-n+1,1)$, the associated convex hypersurface is of class $C^{1}$ (or $C^{1,\gamma}$, respectively) provided that $f \in L^{\infty}(\mathbb{S}^{n})$ (or $C^{0,1}(\mathbb{S}^{n})$, respectively).
\end{enumerate}
\end{remark}

\bigskip

To prove \Cref{thm:main-MP}, we utilize the anisotropic $\alpha$-Gauss curvature flow, defined in \cite{And00}, with $\alpha = \frac{1}{1-p}$. We say that a one-parameter family of complete convex hypersurfaces $\{M_t\}_{t\in I}$ is an {\it anisotropic $\alpha$-Gauss curvature flow} if 
\begin{equation} \label{eq:flow}
\pa_t x = - f(\nu) K^{\alpha} \nu,
\end{equation}
holds $\nu$ and $K$ are the unit outward normal and the Gauss curvature of $M_{t}$ at $x$, respectively. We study local behaviors of the flow \eqref{eq:flow} when $p \le -n+2$, or equivalently $\alpha \in (0, \frac{1}{n-1}]$.

\begin{theorem} \label{thm:main}
Let $\al \in (0, \fr{1}{n-1}]$, and let $M_{0}$ be a smooth, closed, and strictly convex hypersurface in $\mathbb{R}^{n+1}$. Suppose $f$ is a positive smooth function on $\bS^n$. 
If $\{M_t\}_{0\leq t \leq T}$ is the anisotropic Gauss curvature flow \eqref{eq:flow}, then there exists a positive constant $C$ depending only on $n$, $\alpha$, the diameter of $M_{0}$, the inradius of $M_T$, $\|f\|_{C^2(\mathbb{S}^{n})}$, and $\min_{\mathbb{S}^{n}}f$ such that 
\begin{align*}
\max_{i=1,\cdots,n}\la_i(\cdot,t) \le C(1+t^{-\fr{2+\al}{(n\al+1)\al}}) \qu \text{for all} \qu 0< t\le T.
\end{align*}
\end{theorem}

The flow approach has been applied to resolve various Minkowski problems, primarily focusing on asymptotic behavior. For instance, the classical Minkowski problem was studied using logarithmic Gauss curvature flow in \cite{CW00_AIHPANL,Wang96_TAMS}. Moreover, the $L_p$ Minkowski problem was investigated in \cite{BIS19,GLW22,Ivaki16_JFA}; the dual Minkowski problem, an important variant of the logarithmic Minkowski problem, was explored in \cite{LSW20}; and the $L_p$ dual Minkowski problem was researched in \cite{CHZ19_MA,CL21_JFA,GLW23_ME}. The Christoffel-Minkowski problem, which extends the Gauss curvature in the Minkowski problem to $\sigma_k$ curvature, was examined in \cite{BIS23_TAMS}. For additional studies related to the $L_p$ Minkowski problem with $f=1$, we refer the readers to \cite{And99,AC12_PAMQ,AGN16,BCD17,Chow85_JDG,Fir74,GN17,Tso85_CPAM}
 for contraction cases, to \cite{Gerhardt90_JDG,Gerhardt14_CVPDE,Ivaki15_PAMS,Li10_PAMS,Schnurer06_JRAM,Urbas91_JDG} for expansion cases, and to \cite{Stancu12_IMRN} for both types of flows.

It is worth noting that all flows in the aforementioned references, after normalization, have uniform positive lower bounds for their support functions, which makes limiting solitons smooth. 
However, in our range of $p$ with non-symmetric function $f$, the origin may lie on the limiting soliton, and therefore \eqref{eq:MA} becomes a degenerate Monge--Amp\`ere type equation and may not have $C^2$ solutions.

\bigskip

Finding $C^{1,1}$ solutions has attracted much attention as partial regularity theorems of degenerate Monge--Amp\`ere type equations. For the classical degenerate Monge--Amp\`ere equation, $C^{1,1}$ solutions are found in \cite{Guan97_Duke,GTW99_Acta}. Similar results for Minkowski type problem can be found in \cite{CTWX21_CVPDE,GL97_CPAM}. For the flows with degeneracy, we refer readers to \cite{CCD22,CCD24,CDK21,DL04,KLR13} for surfaces with flat sides and \cite{LL21} for surfaces with an obstacle.

\bigskip

The paper is organized as follows. In \Cref{sec:preliminaries}, we provide a brief overview of the $L_{p}$ Minkowski problem and the anisotropic $\alpha$-Gauss curvature flow. We also introduce the concept of viscosity solutions of \eqref{eq:flow} and define some notations. In \Cref{sec:evolution}, we present evolution equations of various geometric quantities under the flow. In \Cref{sec:curvature}, we prove \Cref{thm:main}, which provides a uniform upper bound for the principal curvatures of the flow. We extend the estimate in \Cref{thm:main} to the viscosity solution of \eqref{eq:flow} and present the proof of \Cref{thm:main-MP} in \Cref{sec:thmpf}. The proof of \Cref{thm:counter-MP} is provided in \Cref{sec:thmpf2}.

\bigskip

\noindent\textbf{Acknowledgments.}
KC has been supported by the KIAS Individual Grant MG078902, an Asian Young Scientist Fellowship, and the National Research Foundation (NRF) grants  RS-2023-00219980 and RS-2024-00345403.
MK has been supported by the research fund of Hanyang University HY-202300000001143 and the National Research Foundation of Korea grant RS-2023-00252297.
TL has been supported by National Research Foundation of Korea grant RS-2023-00211258.

\bigskip

\section{Preliminaries} \label{sec:preliminaries}

In this section, we recall precise definitions for the $L_{p}$ Minkowski problem and the anisotropic $\alpha$-Gauss curvature flow, which are closely related to each other. We also make some notations.

\subsection{\texorpdfstring{$L_{p}$}{Lp} Minkowski problem} \label{sec:MP}

The Minkowski problem asks for the construction of a {\it convex body}, a compact convex set in $\mathbb{R}^{n+1}$ with a nonempty interior, whose Gauss curvature is prescribed. It is equivalent to finding a convex body $\Omega$ satisfying
\begin{equation*}
S(\Omega, \cdot) = \mu
\end{equation*}
for a given measure $\mu$ on $\mathbb{S}^{n}$, where $S(\Omega, \cdot)$ is the {\it surface area measure of $\Omega$} determined by the Aleksandrov variational formula
\begin{equation*}
\left. \frac{\mathrm{d} V(\Omega+t\Omega_{0})}{\mathrm{d}t} \right|_{t = 0^{+}} = \int_{\mathbb{S}^{n}} u_{\Omega_{0}}(z) \,\mathrm{d}S(\Omega, z) \quad \text{for any convex body }\Omega_{0}.
\end{equation*}
Here, $V(\Om)$ denotes the volume of $\Om$, and $u_{\Omega_{0}}: \mathbb{S}^{n} \to \mathbb{R}$ is the {\it support function of $\Omega_{0}$} defined by $u_{\Omega_{0}}(z) = \max \lbrace z \cdot x: x \in \Omega_{0} \rbrace$.

\bigskip

To describe the $L_{p}$ Minkowski problem, we recall Firey's \cite{Fir74} {\it $p$-linear combination} of convex bodies: for $p \geq 1$, $t_{1}, t_{2} > 0$, and convex bodies $\Omega_{1}$, $\Omega_{2}$ containing the origin, $t_{1} \cdot_{p} \Omega_{1} +_{p} t_{2} \cdot_{p} \Omega_{2}$ is defined by the convex body with support function $(t_{1}u_{\Omega_{1}}^{p} + t_{2}u_{\Omega_{2}}^{p})^{1/p}$. Then, the {\it $L_{p}$ surface area measure $S_{p}(\Omega, \cdot)$ of a convex body $\Omega$} is characterized by
\begin{equation*}
\left. \frac{\mathrm{d} V(\Omega+_{p}t \cdot_{p} \Omega_{0})}{\mathrm{d}t} \right|_{t = 0^{+}} = \frac{1}{p} \int_{\mathbb{S}^{n}} u_{\Omega_{0}}^{p}(z) \,\mathrm{d}S_{p}(\Omega, z) \quad \text{for any convex body }\Omega_{0}.
\end{equation*}
It is known \cite{Lut93} that $S_{p}(\Omega, \cdot)$ is related to $S(\Omega, \cdot)$ by $S_{p}(\Omega, \cdot) = u_{\Omega}^{1-p} S(\Omega,\cdot)$, and this relation makes it possible to define the $L_{p}$ surface area measure for all $p \in \mathbb{R}$. The $L_{p}$ Minkowski problem is to find a convex body $\Omega$ satisfying
\begin{equation} \label{eq:LpMP}
S_{p}(\Omega, \cdot) = \mu
\end{equation}
for a given measure $\mu$ on $\mathbb{S}^{n}$. If $\mu$ has a density $f$ so that $\mu=f \,\mathrm{d} \si$, then \eqref{eq:LpMP} is equal to \eqref{eq:MA}.

\subsection{Anisotropic \texorpdfstring{$\alpha$}{a}-Gauss curvature flow}

Let us consider an anisotropic Gauss curvature flow \eqref{eq:flow} for $\alpha > 0$ and a positive smooth function $f$ defined on $\bS^n$.
It follows from the result of Andrews \cite{And00} that any closed convex hypersurface converges to a point in a finite time $T_{\max}>0$ under the flow \eqref{eq:flow}.

\begin{thmx} [Andrews \cite{And00}]\label{thm:smooth}
Let $\al>0$ and $f\in C^\infty(\bS^n)$. If $M_0$ is a smooth, closed, strictly convex hypersurface in $\R^{n+1}$, then there exists a unique solution $\{M_t\}_{0<t<T_{\max}}$ to \eqref{eq:flow} starting from $M_0$ such that $M_t$ converges to a point as time $t$ approaches its maximal existence time $T_{\max}<\infty$.
\end{thmx}

Additionally, Andrews \cite{And00} defined viscosity solutions for the flow \eqref{eq:flow} that starts from general convex bodies. We adopt the definition to include generalized solutions to the $L_p$ Minkowski problems.

\begin{definition}[Viscosity Solution \cite{And00}]\label{def:visc-sol}
Let $\{\Om_t\}_{0< t\le T}$ be a family of convex regions, and let $f$ be a positive smooth function on $\bS^n$.
We say that $\{\Om_t\}_{0< t\le T}$ is a viscosity solution to the flow \eqref{eq:flow} if it satisfies the following two conditions:
\begin{enumerate}[(i)]
\item For any strictly convex, smooth hypersurface $M_0^-$ contained in $\Om_0$, the evolution $M_t^-$ of $M_0^-$ under the flow \eqref{eq:flow} is contained in $\Om_t$ up to the maximal existence time of $M_t^-$.
\item For any strictly convex, smooth hypersurface $M_0^+$ enclosing $\Om_0$, the evolution $M_t^+$ of $M_0^+$ under the flow \eqref{eq:flow} encloses $\Om_t$ for all $0\le t\le T$.
\end{enumerate}
We refer to $\{\partial \Omega_t\}_{0<t\le T}$ as a viscosity solution as well, by slight abuse of notation.
\end{definition}

The existence and uniqueness of viscosity solutions are well-established in that paper. 

\begin{thmx} [Andrews, Theorem 15 in \cite{And00}]\label{thm:visco}
Let $\al>0$ and $f\in C^\infty(\bS^n)$. If $\Om_0$ is an open bounded convex region, then there exists a unique viscosity solution $\{\Om_t\}_{0<t<T_{\max}}$, which converges to $\Om_0$ in Hausdorff distance as $t$ approaches zero. Moreover, $\Om_t$ converges to a point as $t$ approaches its maximal existence time $T_{\max}<\infty$.
\end{thmx}

Lastly, we recall the upper bound for the Gauss curvature provided in \cite{And00}. We denote by $\rho_-(M_T)$ the inradius of $M_T$.

\begin{thmx}\label{lem:K}
Let $M_t$ be a strictly convex and smooth solution to the flow \eqref{eq:flow} for $t \in [0,T]$. Then there is a positive constant $C$ depending only on $n$, $\al$, $\max _{\bS^n}f$, $\min_{\bS^n}f$, $\diam(M_0)$, and $\rho_-(M_T)$ such that 
\begin{align*}
K(\cdot, t) \le C(1+t^{-\fr{n}{n\al+1}}) \qu \text{for all}\qu 0< t\le T.
\end{align*}
\end{thmx}

\bigskip

\subsection{Notations}
Let $\Sigma$ be  a strictly convex hypersurface in $\R^{n+1}$. Then, we denote by $\{g_{ij}\}$ and $\{h_{ij}\}$ the induced metric and the second fundamental form, respectively. The inverse of the metric is denoted by $g^{ij}$. Since $\Sigma$ is strictly convex, $\{h_{ij}\}$ has the inverse matrix $\{b^{ij}\}$, namely 
\begin{equation*}
    b^{ik}h_{kj}=\delta_{ij},
\end{equation*}
where $\delta_{ij}$ is the Kronecker delta. Also, we denote 
\begin{equation*}
\operatorname{tr}(b)=b^{ij}g_{ij}.    
\end{equation*}
 
\bigskip

Given a function $f\in C^\infty(\mathbb{S}^n)$, we may extend its domain to $\R^{n+1}\sm\{0\}\to \mathbb{R}$ so that by abuse of notation we define $f(z):=f(z/|z|)$ for $z\neq 0$. We denote by $Df$ the gradient of $f$ with respect to the standard euclidean metric, and denote by $D_zf=\inn{Df}{z}$ the directional derivative of $f$ in the direction $z$. For the purpose of brevity, we define
\begin{align*}
    & f_i=D_{\D_ix}f, && f_{ij}=D^2_{\D_ix,\D_jx}f
\end{align*}
for $\D_ix,\D_jx\in T_x\Sigma$, and
\begin{align*}
    & f^i=g^{ij}f_j, && f^{ij}=g^{ik}g^{ij}f_{kl}.
\end{align*}

\bigskip

\begin{remark}
The indices of $f_i,f_{ij}$ are determined by local charts of $\Sigma$ rather than by the fixed coordinates of the ambient space $\mathbb{R}^{n+1}$. Namely, $f_i$ would be different from $\frac{\partial}{\partial z_i}f(z)$.  

On the other hand, we can observe $|\bar \nabla_{\mathbb{S}^n} f|^2=f_if^i$ and $|\bar \nabla_{\mathbb{S}^n}^2 f|^2=f_{ij}f^{ij}$, where $\bar \nabla_{\mathbb{S}^n}$ is the Levi-Civita connection of the unit sphere.
\end{remark}

\bigskip

\begin{proposition}\label{prop:Df}
The function $f(\nu)$ of the outward point unit normal vector $\nu$ satisfies 
    \begin{align*}
      &  \nabla_i f(\nu)=h_{ik}f^k, && \nabla_i\nabla_jf(\nu)=f^k\nabla_kh_{ij}+h_{ik}h_{jl}f^{kl}.
    \end{align*}
\end{proposition}

\begin{proof}
    We directly compute $\nabla_i f=  \langle \nabla_i \nu , Df \rangle =h_{ij}g^{jk}\langle \nabla_kx , Df \rangle=h_{ij}f^j$. We can obtain the second identity in the same manner.
\end{proof}

We recall the linearized operator 
\begin{equation*}
    \cL=\al fK^\al b^{ij}\D_{i}\D_j+K^\al f^k\D_k,
\end{equation*}
and the associated inner product
\begin{align*}
\inn{\D v}{\D w}_W=\al fK^\al b^{ij} \D_iv\D_jw
\end{align*}
for $v,w\in C^1(\Sigma)$, where
\begin{align*}
\D v = \D_iv\D^ix.
\end{align*}
Furthermore, we denote the support function and the speed of the flow \eqref{eq:flow} by
\begin{align*}
    & u=\langle x,\nu\rangle, &&  F=fK^\al,
\end{align*}
respectively.

\bigskip

\section{Evolution equations} \label{sec:evolution}

In this section we obtain evolution equations for geometric quantities under the flow \eqref{eq:flow}. The proofs are standard, and some of them can be found in \cite{And00}, where the calculations are performed in the coordinates of $\bS^n$ with respect to the parametrization $x\circ\nu^{-1}:\bS^n\ra \R$ (see also \cite{GN17} for the case $f\equiv 1$). We derive equations by using the covariant derivative on the hypersurface $M_t$.

\begin{lemma}\label{lem:ev}
Given a solution to the flow \eqref{eq:flow}, the following equations hold:
\begin{enumerate}[\normalfont(i)]
\item $\pa_t g_{ij}=-2F h_{ij}$,
\item $\pa_t \nu=\cL\nu + \al F H\nu$,
\item $\pa_t |x|^2=\cL |x|^2 +2(n\al -1)F u-2\al F \operatorname{tr}(b)-2K^\al f^k\langle x,\nabla_k x\rangle$,
\item
\makebox[0.8cm]{$\pa_t h_{ij}$}
$=\cL h_{ij} +(\al^2 b^{pq}b^{kl}-\al b^{kp}b^{lq})F\D_ih_{pq}\D_jh_{kl} +h_{ik}f^k\D_jK^\al+h_{jk}f^k\D_iK^\al$ \\
\makebox[1.5cm]{ \hfill}
$+K^\al  f^{kl} h_{ik}h_{jl}-(n\al+1)Fh_{ik}h^k_{j}+\al F Hh_{ij} $,
\item $\pa_t F=\cL F   +\al F^2H$.
\end{enumerate}

\end{lemma}

\begin{proof}
(i) Since $\inn{\pa_jx}{\nu}=0$, we have 
\begin{align*}
\inn{\pa_t\pa_ix}{\pa_jx}=\inn{\pa_i\pa_t x}{\pa_jx}=\inn{\pa_i(-F\nu)}{\pa_jx}=-F h_{ij}.
\end{align*} 
Similarly, we obtain $\inn{\pa_ix}{\pa_t\pa_jx}=-F h_{ij}$ and thus 
\begin{align*}
\pa_t g_{ij}=\inn{\pa_t\pa_ix}{\pa_jx}+\inn{\pa_ix}{\pa_t\pa_jx}=-2F h_{ij}
\end{align*}
as desired.

\bigskip

(ii) We first compute 
\begin{align*}
\inn{\pa_t \nu}{\pa_k x}=-\inn{\nu}{\pa_k\pa_t x}=-\inn{\nu}{\pa_k(-F \nu)}=\pa_kF.
\end{align*} Thus we have 
\begin{equation}\label{eq:nu_t}
 \pa_t \nu = \D F.   
\end{equation}
On the other hand, observing
\begin{equation*}
    \nabla_i\nabla_j \nu=\nabla_i(h_{jk}\nabla^kx)=\nabla_kh_{ij}\nabla^kx-h_{i}^kh_{jk}\nu,
\end{equation*}
we have
\begin{align*}
\al Fb^{ij}\nabla_i\nabla_j\nu= f\D K^\al -\al F H\nu.
\end{align*}
Also, Proposition \ref{prop:Df} implies
\begin{equation*}
    K^\al f^k\D_k \nu=K^\al f^kh_{kl} \D^lx =K^\al \langle \D_l f,\D^l x\rangle =K^\al \nabla f.
\end{equation*}
 Remembering $F=f K^\al$ and the definition of $\cL$, the equations above give us 
\begin{align*}
(\pa_t -\cL)\nu = \D(fK^\al )- f\D K^\al +\al FH\nu- K^\al \D f=\al F H\nu.
\end{align*}

\bigskip

(iii) Note that 
\begin{align*}
\pa_t |x|^2
&=2\inn{x}{ x_t}=-2F u,
\\
\D_i\D_j|x|^2
&=2 	\inn{\D_ix}{\D_j x}+2\inn{x}{ \D_i\D_jx}=2g_{ij}-2h_{ij}u,
\\
\al F b^{ij}\D_i\D_j|x|^2
&=2\al F  \operatorname{tr}(b)-2n\al F u,
\\
K^\al f^k\D_k|x|^2&=2K^\al f^k\langle x,D_kx\rangle.
\end{align*}
Thus
\begin{align*}
(\pa_t -\cL)|x|^2 =-2F u-2\al F\operatorname{tr}(b)+2n\al F u-2K^\al D_xf.
\end{align*}

\bigskip

(iv) We note that
\begin{align*}
\D_j \D F= \D_j(\D^kF\D_kx)=\D_j\D^kF\D_kx-(\D^kF) h_{kj}\nu,
\end{align*} 
and thus remembering $\nu_t=\nabla F$ in \eqref{eq:nu_t} we have
\begin{align*}
\pa_t h_{ij} 
&= \inn{\D_i\pa_t x}{\D_j\nu}+\inn{\D_ix}{\D_j\pa_t \nu}= \inn{\D_i(- F \nu)}{\D_j\nu}+\inn{\D_ix}{\D_j\D F}
\\
&=-F h_{ik}h^k_{j}+ \D_i\D_j F.
\end{align*}
Using the Gauss equation $R_{ijkl}=h_{ik}h_{jl}-h_{il}h_{jk}$ and the Codazzi equation, we obtain
\begin{align*}
\D_k\D_lh_{ij}
&= \D_k\D_ih_{jl}= \D_i\D_kh_{jl}+R_{kijm}h^m_l+R_{kilm}h_j^m
\\
&=\D_i\D_jh_{kl}+(h_{kj}h_{im}-h_{km}h_{ij})h^m_l+(h_{kl}h_{im}-h_{km}h_{il})h_j^m,
\end{align*}
and therefore
\begin{align*}
\al F b^{kl}\D_k\D_l h_{ij}
=\al F b^{kl}\D_i\D_jh_{kl}+\al F (-Hh_{ij}+nh_{im}h^m_{j}).
\end{align*}
Hence
\begin{align*}
\pa_t h_{ij}&= \al F b^{kl}\D_k\D_lh_{ij}-(n\al+1)F h_{ik}h^k_{j}+ \D_i\D_jF-\al F b^{kl}\D_i\D_jh_{kl}+\al F Hh_{ij}.
\end{align*}
By using $\D_jK^\al =\al K^\al b^{kl}\D_jh_{kl}$ and 
\begin{align*}
\D_i\D_jK^\al =\al K^\al b^{kl}\D_i\D_jh_{kl}+\al \D_i(K^\al b^{kl})\D_jh_{kl},
\end{align*}
we obtain
\begin{align*}
\D_i\D_jF-\al F b^{kl}\D_i\D_jh_{kl}&=K^\al \D_i\D_jf
+\D_i f\D_j K^\al +\D_j f\D_i K^\al \\ &\qu+\al f\D_i(K^\al b^{kl})\D_jh_{kl}.    
\end{align*}
Thus, combining above equations yields 
\begin{align*}
\pa_t h_{ij}
&= \al F b^{kl}\D_k\D_l h_{ij}+(\al^2b^{pq}b^{kl}-\al b^{kp}b^{lq})F\D_ih_{pq}\D_jh_{kl}+K^\al \D_i\D_jf
\\&\qu+\D_if\D_jK^\al +\D_jf\D_iK^\al-(n\al+1)Fh_{ik}h^k_{j}+\al F Hh_{ij}.
\end{align*}
Then, applying Proposition \ref{prop:Df} completes the proof.

\bigskip

(v) Since $\inn{Df}{ \nu}=0$, the equation (ii) implies $\inn{Df}{(\pa_t-\cL)\nu}=0$. Thus, observing
\begin{align*}
\al F b^{ij}\D_i\nu\D_j\nu=\al F b^{ij}h_{ik}h_{jl}\D^kx\D^lx=\al F h_{kl}\D^kx\D^lx,
\end{align*}
we have
\begin{equation*}
 (\pa_t -\cL)f=\inn{Df}{ (\pa_t -\cL)\nu}-\al F b^{ij}D^2f\D_i\nu\D_j\nu=-\al F f^{kl}h_{kl}.   
\end{equation*}
  
 \bigskip
 
To derive the equation for $K^\al $, we first check 
\begin{align}\label{eq:ijg_t}
\pa_t g^{ij} = -g^{ik}g^{jl}\pa_t g_{kl}=-2g^{ik}g^{jl}(-F h_{kl})=2F h^{ij}.
\end{align} 
Since $K=\det (g^{ik}h_{kj})$, we have
\begin{align}\label{eq:ev-K^a-1}
(\pa_t-\cL) K^\al
&= \al K^\al g_{ij}\pa_t g^{ij}+\al K^\al b^{ij}(\pa_t-\cL) h_{ij}-\al F b^{kl}\D_k (K^\al b^{ij})\D_l h_{ij}.
\end{align}
Using \eqref{eq:ijg_t}, the first term becomes 
\begin{align*}
\al K^\al g_{ij}\pa_t g^{ij}=2\al K^{\al}FH.
\end{align*}
For the other two terms, we see that
\begin{align*}
\al K^\al b^{ij}(\al^2b^{pq}b^{kl}-\al b^{kp}b^{lq})F\D_ih_{pq}\D_jh_{kl}&=\al F b^{kl}\D_k (K^\al b^{ij})\D_l h_{ij},
\\
\al K^\al b^{ij}(-(n\al+1)Fh_{ik}h^k_{j}+\al FHh_{ij})&= -\al K^\al FH.
\end{align*}
Hence, \eqref{eq:ev-K^a-1} becomes 
\begin{align*}
\pa_t K^\al 
&= \cL K^\al+\al K^\al b^{ij}(2h_{ik}f^k\D_jK^\al+K^\al  f^{kl} h_{ik}h_{jl} )+\al K^\al FH\\
&= \cL K^\al+2\al K^\al f^k\D_kK^\al+\al K^{2\al}f^{kl}h_{kl}
+\al K^{\al}FH.
\end{align*}

Combining these equations for $f$ and $K^\al$, we obtain for $F=fK^\al$ that
\begin{align*}
(\pa_t-\cL)F&=f(\pa_t-\cL)K^\al+K^\al(\pa_t-\cL)f-2\al F b^{ij}\nabla_if\D_jK^\al
\\
&=2\al F f^k\D_kK^\al+\al K^{\al}Ff^{kl}h_{kl}+\al F^2H
\\
&\qu-\al K^{\al}F f^{kl}h_{kl}-2\al F b^{ij}h_{ik}f^k\D_jK^\al
\\
&=\al F^2H.
\end{align*}
This completes the proof.
\end{proof}

\section{Principal curvature bound} \label{sec:curvature}
In this section we provide a uniform upper bound for principal curvatures. The estimate heavily relies on diameter and the upper bound on the Gauss curvature $K$ in \Cref{lem:K}. 	Since the minimum eigenvalue is not necessarily smooth, we need the following derivatives of a smooth approximation for the minimum eigenvalue.

\begin{lemma}\label{lem:vp}
Let $\mu$ be the multiplicity of the smallest principal curvature at a point $x_0$ on ${M}_{t_0}$ for $t_0>0$ so that $\la_1=\cdots=\la_\mu<\la_{\mu+1}\le \cdots\le \la_n$, where $\lambda_1,\cdots,\lambda_n$ are the principal curvatures. Suppose $\vp$ is a smooth function defined on $\mathcal{M}=\cup_{0<t\leq t_0} M_t \times \{t\}$ such that $\vp\le \la_1$ on $\mathcal{M}$ and $\vp=\la_1$ at $X_0=(x_0,t_0)$. Then, at the point $X_0$ with a chart of $g_{ij}=\de_{ij}$ and $h_{ij}=\la_i\de_{ij}$, we have $\D_i h_{kl}=\D_i\vp \de_{kl}$ for $1\le k,l\le \mu$,
\begin{align}\label{eq:vp}
\D_i\D_i\vp \le \D_i\D_i h_{11} -\sum_{l>\mu}\fr{2(\D_ih_{1l})^2}{\la_l-\la_1}, 
\qu \text{and}\qu 
\pa_t \vp \ge\pa_t h_{11}  +2F \la_1^2.
\end{align}
\end{lemma}

\begin{proof}
The identity for $\D \vp$ and the inequality  for  $\D^2 \vp$ are already given in \cite{BCD17}. To prove the inequality for $\pa_t \vp$, we note that $\la_1\le h_{11}/g_{11}$. Indeed, if $\{ E_1,\cdots,E_n\}$ is an orthonormal basis of the tangent space $T(M_{t_0})_{x_0}$ such that $L(E_j)=\la_jE_j$ for the Weingarten map $L$, then it follows by writing $\D_1x=a_jE_j$ that 
\begin{align*}
h_{11}= \inn{L(\D_1x)}{\D_1x}=\sum_{j=1}^n\la_ja_j^2\ge \la_1\sum_{j=1}^na_j^2=\la_1g_{11}.
\end{align*}

Thus, $h_{11}/g_{11}-\vp\ge0$ and the minimum occurs at $X_0$. Taking derivative in time, we obtain
\begin{align*}
0\ge \fr{\mathrm{d}}{\mathrm{d}t}\bigg\vert_{t=t_0}\left(\fr{h_{11}}{g_{11}}-\vp\right)=\fr{\pa_t h_{11}}{g_{11}}-\fr{h_{11}\pa_t g_{11}}{g_{11}^2}-\pa_t \vp.
\end{align*}
Then the conclusion follows from $g_{11}=1$, $h_{11}=\la_1$, and $\pa_t g_{11}=-2Fh_{11}=-2F\la_1$.
\end{proof}

\bigskip

\begin{lemma}\label{lem:fKvp}
Let $M_t$ be a strictly convex and smooth solution to the flow \eqref{eq:flow} for $t \in [0,T]$, and let $\vp$ be the function satisfying the conditions in \Cref{lem:vp}. Then the following holds at the maximum point of $\vp$:
\begin{align*}
\pa_t(F\vp^{-1})
&\le \cL(F\vp^{-1})
-\fr{\al F^2}{ \la_1^{4}}(\D_1 h_{11})^2
-\sum_{k>\mu}\fr{2\al F^2 (\D_kh_{11})^2}{\la_1^3(\la_k-\la_1)}
\\
&\qu 
-\fr{(\D_1F)^2}{\la_1^2}
-2\inn{\D F}{\D \vp^{-1}}_W
+CF^2,
\end{align*}
where $C$ is a constant depending only on $n$, $\al$, $\|f\|_{C^2(\mathbb{S}^n)}$, and $\min_{\mathbb{S}^{n}}f$,.
\end{lemma}

\begin{proof}
From the evolution equation of $h_{ij}$ in \Cref{lem:ev}, we have
\begin{align*}
(\pa_t-\cL)h_{11}&=(\al^2 b^{pq}b^{kl}-\al b^{kp}b^{lq})F\D_1h_{pq}\D_1h_{kl}+2h_{1k}f^k \D_1K^\al +K^\al  f^{kl} h_{1k}h_{1l}
\\&\qu-(n\al+1)Fh^k_1h_{k1}+\al F Hh_{11}.
\end{align*}
By \Cref{lem:vp}, we obtain at the point
\begin{align*}
(\pa_t-\cL) \vp &\ge (\pa_t-\cL)h_{11}+\sum_k\sum_{l>\mu}\fr{2\al F(\D_kh_{1l})^2}{\la_k(\la_l-\la_1)}+2 F \la_1 ^2
\\
&= (\al^2 b^{pq}b^{kl}-\al b^{kp}b^{lq})F\D_1h_{pq}\D_1h_{kl}
+\sum_k\sum_{l>\mu}\fr{2\al F (\D_kh_{1l})^2}{\la_k(\la_l-\la_1)}
+2 f_1\lambda_1\D_1K^\al
\\
&\qu+K^\al f_{11}\la_1^2-(n\al-1)F \la_1^2+\al F H\la_1 .
\end{align*}
Since $\D_1K^\al=\al K^\al b^{ij}\D_1h_{ij}$, we have
\begin{align*}
\al^2 b^{pq}b^{kl}F\D_1h_{pq}\D_1h_{kl}
=
fK^{-\al}(\D_1  K^\al )^2.
\end{align*}
On the other hand, by \Cref{lem:vp} and the Codazzi equation, we get $\D_1h_{ik}=0$ for $1< k\le \mu$, which implies
\begin{align*}
\al b^{kp}b^{lq}F\D_1h_{pq}\D_1h_{kl}
=\frac{\al F (\D_1h_{11})^2}{\la_1^{2}}
+\sum_{k>\mu}\frac{2\al F (\D_kh_{11})^2}{\la_1\la_k}+\sum_{k,l>\mu}\frac{\al F (\D_1h_{kl})^2}{\la_k\la_l},
\end{align*}
and
\begin{align*}
\sum_k\sum_{l>\mu}\fr{2\al F (\D_kh_{1l})^2}{\la_k(\la_l-\la_1)}
=\sum_{k>\mu}\fr{2\al F (\D_kh_{11})^2}{\la_1(\la_k-\la_1)}+\sum_{k,l>\mu}\fr{2\al F (\D_1h_{kl})^2}{\la_k(\la_l-\la_1)}.
\end{align*}
Hence,
\begin{align*}
&-\al b^{kp}b^{lq}F\D_1h_{pq}\D_1h_{kl}+\sum_k\sum_{l>\mu}\fr{2\al F (\D_kh_{1l})^2}{\la_k(\la_l-\la_1)}\\
&\geq -\frac{\al F (\D_1h_{11})^2}{\la_1^{2}}+\sum_{k>\mu}\fr{2\al F (\D_kh_{11})^2}{\la_k(\la_k-\la_1)}.
\end{align*}
Moreover, using Proposition \ref{prop:Df} and $F=fK^\al$, we have
\begin{align*}
fK^{-\al}(\D_1K^\al)^2+2\lambda_1 f_1\D_1K^\al= \fr{(\D_1 F)^2}{F}-\fr{K^\al (\lambda_1 f_1)^2}{f}=\fr{(\D_1 F)^2}{F}-\fr{F \la_1^2 f_1^2}{f^2}.
\end{align*}
Then we have
\begin{align*}
(\pa_t-\cL) \vp 
&\ge 
\fr{(\D_1F)^2}{F}-\al F \la_1^{-2}(\D_1h_{11})^2
+\sum_{k>\mu}\fr{2\al F (\D_kh_{11})^2}{\la_k(\la_k-\la_1)}
\\
&\qu +F\la_1^2\left(\fr{f_{11}}{f}-\fr{f_1^2}{f^2}-n\al+1\right)+\al F H\la_1.
\end{align*}

Observing $\D_j\vp^{-1}=-\vp^{-2}\D_j\vp$ and $\D_i\D_j\vp^{-1}=-\vp^{-2}\D_i\D_j\vp+2\vp^{-3}\D_i\vp\D_j\vp$, we have 
\begin{align*}
(\pa_t-\cL)\vp^{-1}=-\la_1^{-2}(\pa_t-\cL)\vp-\sum_{k=1}^n\fr{2\al F(\D_kh_{11})^2}{\la_1^3\la_k}.
\end{align*}
Since
\begin{align*}
-\sum_{k>\mu}\fr{2\al F (\D_kh_{11})^2}{\lambda_1^2 \la_k(\la_k-\la_1)}-\sum_{k=2}^n\fr{2\al F(\D_kh_{11})^2}{\la_1^3\la_k}
\le 
-\sum_{k>\mu}\fr{2\al F(\D_kh_{11})^2}{\la_1^3(\la_k-\la_1)},
\end{align*}
combining with \eqref{eq:vp} yields
\begin{align*}
(\pa_t-\cL)\vp^{-1}
&\le 
-\fr{(\D_1F)^2}{F\la_1^2}
-\al F \la_1^{-4}(\D_1h_{11})^2
-\sum_{k>\mu}\fr{2\al F(\D_kh_{11})^2}{\la_1^3(\la_k-\la_1)}
\\
&\qu
-F\left(\fr{f_{11}}{f}-\fr{f_1^2}{f^2}-n\al+1\right)
-\al F H\la_1^{-1}.
\end{align*}
Hence, by using the evolution equation of $F$ in Lemma \ref{lem:ev}, we have
\begin{align*}
\pa_t(F \vp^{-1})
&= \cL(F\vp^{-1})
+F(\pa_t -\cL)\vp^{-1}+\vp^{-1}(\pa_t -\cL)F-2\inn{\D F}{\D \vp^{-1}}_W
\\
&\le
\cL(F\vp^{-1})-\fr{(\D_1F)^2}{\la_1^2}
-\fr{\al F^2}{ \la_1^{4}}(\D_1h_{11})^2
-\sum_{k>\mu}\fr{2\al F^2(\D_kh_{11})^2}{\la_1^3(\la_k-\la_1)}
\\
&\qu
-F^2\left(\fr{f_{11}}{f}-\fr{f_1^2}{f^2}-n\al+1\right)
-2\inn{\D F}{\D\vp^{-1}}_W.
\end{align*}
Then the conclusion follows from
\begin{align*}
\fr{f_{11}}{f}-\fr{f_1^2}{f^2}-n\al+1 \ge -C\left(n,\al, \|f\|_{C^2(\mathbb{S}^n)}, \min_{\mathbb{S}^{n}}f \right).
\end{align*}
\end{proof}

We are now ready to establish a uniform upper bound for principal curvatures.

\begin{proof}[Proof of \Cref{thm:main}]
Since $M_t$ shrinks under the flow \eqref{eq:flow}, we have
\begin{align*}
R:=4\sup_{ 0\le t\le T}\sup_{ M_t}|x|^2=4\sup_{ M_0}|x|^2<\infty.
\end{align*}
We consider a continuous function
\begin{align*}
 w := \fr{tfK^\al\la_{\min}^{-1}}{R-|x|^2}.
\end{align*}
Since $w(\cdot,0)=0$, we can take its maximum point $(x_0,t_0)$ such that $w\leq w_0:=w(x_0,t_0)$ holds for $t\leq t_0$. Clearly, $t_0>0$.

\bigskip

We claim that $w\le C(1+t)$ holds for some $C$ depending only on $\al$, $n$, $T$, $\|f\|_{C^2(\mathbb{S}^n)}$, $\min_{\mathbb{S}^n}f$, and $R$. If the claim is true, then by using $K^\al \la_{\min}^{-1}\ge \la_{\min}^{(n-1)\al-1}\la_{\max}^\al$, $\al \le 1/(n-1)$, and \Cref{lem:K}, we can obtain the desired result as follow. 
\begin{align*}
\la_{\max}^\al \le C(1+t^{-1}) \la_{\min}^{1-(n-1)\al}\le C(1+t^{-1})K^{\fr{1-(n-1)\al}{n}}\le C(1+t^{-1})^{\fr{2+\al}{n\al+1}}.
\end{align*} 
Note that we can drop the dependence on $T$, because we can obtain $T\leq C(\al,n,\text{diam} M_0, \min_{\mathbb{S}^n}f)$ by applying the comparison principle with a shrinking ball satisfying $x_t=- (\min f)K^\al \nu$.

\bigskip

To prove the claim, we define
\begin{equation*}
    \vp:=\fr{t fK^\al w_0^{-1}}{R-|x|^2}.
\end{equation*}
Then, observing $w\leq w_0$ we have $\vp\le \la_{\min}$ on $M_t$ for $t\in [0,t_0]$ as well as $\vp(x_0,t_0)=\la_{\min}(x_0,t_0)$. Hence, the smooth function $\vp$ satisfies the assumptions in \Cref{lem:vp}. We choose a chart satisfying $g_{ij}=\de_{ij}$ and $h_{ij}=\la_i\de_{ij}$ at $(x_0,t_0)$, where $\la_1=\cdots=\la_\mu<\la_{\mu+1}\leq \cdots \leq \la_n$. Then, by \Cref{lem:fKvp} we have
\begin{align*}
    \pa_t(F\vp^{-1})
&\le \cL(F\vp^{-1})
-\fr{\al F^2}{ \la_1^{4}}(\D_1 h_{11})^2
-\sum_{k>\mu}\fr{2\al F^2 (\D_kh_{11})^2}{\la_1^3(\la_k-\la_1)}
\\
&\qu 
-\fr{(\D_1F)^2}{\la_1^2}
-2\inn{\D F}{\D \vp^{-1}}_W +CF^2
\end{align*}
at $(x_0,t_0)$. Since $F\varphi^{-1}:=w_0t^{-1}(R-|x|^2)$, the equation of $|x|^2$ in \Cref{lem:ev} implies
\begin{align*}
0 &\le \frac{w_0}{t_0^2}(R-|x_0|^2)+\frac{2(n\al -1)w_0}{t_0}F u-\frac{2\al w_0}{t_0} F \operatorname{tr}(b)-\frac{2w_0}{t_0}K^\al f^k\langle x_0,\nabla_k x\rangle
\\
&\qu -\fr{\al F^2}{ \la_1^{4}}(\D_1 h_{11})^2
-\sum_{k>\mu}\fr{2\al F^2 (\D_kh_{11})^2}{\la_1^3(\la_k-\la_1)} -\fr{(\D_1F)^2}{\la_1^2}
-2\inn{\D F}{\D \vp^{-1}}_W +CF^2.
\end{align*}
at $(x_0,t_0)$. We define $I_k=F\la_1^{-2}\D_kh_{11}$ for $k=1,\dots,n$.
Then, we multiply the above inequality by $t_0^2$ to derive
\begin{align*}
0 &\le (R-|x|^2)w_0+2t_0(n\al -1)F uw_0-2\al t_0 F \operatorname{tr}(b)w_0-2t_0 K^\al f^k\langle x_0,\nabla_k x\rangle w_0\\
&\qu -\al t_0^2 I_1^2
-\sum_{k>\mu}\fr{2\al t_0^2 \la_1 I_k^2}{\la_k-\la_1} -\fr{t_0^2(\D_1F)^2}{\la_1^2} -2t_0^2\inn{\D F}{\D \vp^{-1}}_W  +Ct_0^2F^2.
\end{align*}

Note that we have $|x|\leq R^{\frac{1}{2}}$ and $u\leq R^{\frac{1}{2}}$. Also, \Cref{lem:K} says $tF\le C$. Hence, 
\begin{align*}
0 &\le -2\al t_0 F \operatorname{tr}(b)w_0 -\al t_0^2 I_1^2
-\sum_{k>\mu}\fr{2\al t_0^2\la_1 }{\la_k-\la_1}I_k^2\\&\qu -\fr{t_0^2(\D_1F)^2}{\la_1^2} -2t_0^2\inn{\D F}{\D \vp^{-1}}_W  +Cw_0+C.
\end{align*}

To replace $\D \vp^{-1}$ and $\D F $, remembering $\D_kh_{11}=F^{-1}\la_1^2I_k$, we observe 
\begin{align*}
&\D \vp^{-1}=-\la_1^{-2}\D h_{11}=-F^{-1}I_k\nabla_k x, && \D (t_0F \vp^{-1})= -w_0\D |x|^2.
\end{align*}
Therefore, we have 
\begin{equation*}
   t_0\la_1^{-1} \D F=t_0I_k \nabla_k x- w_0\D|x|^2.
\end{equation*}
This implies
\begin{equation*}
   -t_0^2\la_1^{-2} (\D_1 F)^2= -t_0^2I_1^2 +2t_0 w_0 (\D_1 |x|^2) I_1 - w_0^2|\D_1|x|^2|^2,
\end{equation*}
and
\begin{align*}
-2t_0^2\inn{\D F}{\D \vp^{-1}}_{W}&=2t_0\inn{t_0\la_1 I_k \nabla_k x - \la_1w_0\D|x|^2}{ F^{-1}I_k\nabla_k x}_{W}
\\
&= \sum_{k=1}^n\fr{2\al t_0^2\la_1}{\la_k}I_k^2
- \sum_{k=1}^n\fr{2\al t_0 \la_1 \D_k|x|^2}{\la_k}w_0 I_k.
\end{align*}
Since Lemma \ref{lem:vp} implies $I_k=0$ for $1<k\le\mu$, we have
\begin{equation}\label{eq:Q1}
\begin{split}
0&\le -2\al t_0 F \operatorname{tr}(b)w_0  + (\al-1)t_0^2 I_1^2
-\sum_{k>\mu}\left(\fr{2\al t_0^2 \la_1^{2}}{\la_k(\la_k-\la_1)}I_k^2
+\fr{2\al t_0 \la_1 \D_k|x|^2}{\la_k}w_0 I_k \right)
\\
&\qu +2(1-\al)t_0 w_0(\D_1|x|^2)I_1 - w_0^2|\D_1|x|^2|^2 +Cw_0+C.
\end{split}
\end{equation}

Since $0<\al \leq \frac{1}{n-1}\leq 1$, we have
\begin{align*}
&(\al-1)t_0^2I_1^2+2(1-\al)t_0w_0\D_1|x|^2I_1
\leq (1-\al)w_0^2|\D_1|x|^2|^2.
\end{align*}
Also, it follows from $\la_k-\la_1\le \la_k$ that 
\begin{align*}
-\sum_{k>\mu}\left(\fr{2\al t_0^2 \la_1^2}{\la_k(\la_k-\la_1)}I_k^2
+\fr{2\al t_0 \la_1 \D_k|x|^2 }{\la_k}w_0I_k\right)
\le
\sum_{k>\mu}\fr{\al }{2}w_0^2|\D_k|x|^2|^2.
\end{align*}
Furthermore, by using $\operatorname{tr}(b)=\sum_{k=1}^n \la_k^{-1}\ge \la_1^{-1}$, we have
\begin{align*}
-2\al t_0 F \operatorname{tr}(b)w_0\le -2\al (R-|x|^2)w_0^2.
\end{align*}
Putting these all together, \eqref{eq:Q1} becomes
\begin{align*}
0&\le 
\Big(-2\al(R-|x|^2)+\sum_{k>\mu}\fr{\al}{2}|\D_k|x|^2|^2-\al |\D_1|x|^2|^2\Big) w_0^2
+Cw_0+C.
\end{align*}
Since $|x|^2-u^2=\sum_{k=1}^n\inn{x}{\D_kx}^2$, the coefficient of $w_0^2$ satisfies
\begin{align*}
-2\al(R-|x|^2)+\sum_{k>\mu}\fr{\al}{2}|\D_k|x|^2|^2-\al |\D_1|x|^2|^2
\le
-2 \al (R-|x|^2)+2\al |x|^2
\le
-\al R
\end{align*}
by our choice of $R=4\sup_{ 0\le t\le T}\sup_{ M_t}|x|^2$. Therefore, we arrive at
\begin{align*}
0\leq -\al R w_0^2 +Cw_0+C.
\end{align*}
This completes the proof.
\end{proof}

\bigskip

\section{Proof of \Cref{thm:main-MP}} \label{sec:thmpf}

We first extend the result of \Cref{thm:main} to viscosity solutions using an approximation argument (see \Cref{def:visc-sol} for the definition of viscosity solutions).

\begin{theorem}\label{thm:main-visc}
Let $\alpha\in(0,\frac{1}{n-1}]$, and let $M_0=\pa\Om_0$ be the boundary of a convex body $\Om_0$ in $\mathbb{R}^{n+1}$. Suppose $f$ is a positive smooth function on $\mathbb{S}^n$.
Then, there exists a unique viscosity solution $\{\Om_t\}_{0<t\leq T'}$ to the flow \eqref{eq:flow} which converges to $\Om_0$ in Hausdorff distance as $t\to 0$.  
For each $t\in (0,T']$, the corresponding hypersurface $M_t=\pa\Om_t$ is of class $C^{1,1}$. Moreover, there exist constants $C,T$ depending only on $n$, $\alpha$, $\operatorname{diam}M_{0}$, $\rho_-(M_0)$, $\|f\|_{C^2(\mathbb{S}^{n})}$, and $\min_{\mathbb{S}^{n}}f$ such that
\begin{align}\label{ineq:sup-la}
\operatorname{ess\,sup}_{M_t}\la_i \le C(1+t^{-\fr{1+\al}{n\al^2}}), \qu\text{for}\qu 0<t\le T, \qu i=1,\ldots,n.
\end{align}
\end{theorem}

\begin{proof}
By \Cref{thm:visco}, it suffices to prove \eqref{ineq:sup-la}. Consider an increasing family of smooth, closed, strictly convex hypersurfaces $M^\e$ that converge to $M_0$ in Hausdorff distance as $\e \ra0$. By \Cref{thm:smooth}, there exists a unique solution $M^\e_t$ to \eqref{eq:flow} with initial data $M^\e$.

We choose an origin for $\R^{n+1}$ and radii $0<r<R$ such that, for sufficiently small $\e$, the hypersurfaces $M^\e$ contain the ball $B_r(0)$ and are enclosed by the ball $B_R(0)$. Since $M^\e_t$ shrinks, every $M^\e_t$ with $t \geq 0$ is enclosed by $B_R(0)$. Also, there is a small time $T'>0$ such that all hypersurfaces $M_t^\e$ for $t\in[0,T']$ enclose the ball $B_{r/2}(0)$, by the comparison principle with a shrinking ball $B_{\rho(t)}(0)$ where $\rho'=-(\max_{\mathbb{S}^n}f)\rho^{-n\al}$.

We recall the identity $|x|^2=u^2+|\bar\D_{\mathbb{S}^n} u|^2$ of convex hypersurface, where $x$ is the position vector, $u(\nu)$ is the support function, and $\bar\D_{\mathbb{S}^n}$ is the Levi-Civita connection of the unit sphere metric. Since the hypersurfaces $M^\e$ are convex and have bounded diameter, the identity implies that the gradients of their support functions $u^\e(z,t)$ are uniformly Lipschitz on $\bS^n$. Moreover, according to \Cref{lem:K}, the speeds are bounded over $[\de,T]$ for any $\de>0$. Thus, the functions $u^\e$ are uniformly H\"older continuous on the space-time $\bS^n\ti [\de,T]$. By the Arzela--Ascoli theorem, we can extract a subsequence $\e_k$ such that $u^{\e_k}$ converges uniformly to a limit function $u:\bS^n\ti(0,T]\ra \R$ with the same estimates. By the Blaschke selection theorem, each function $u(\cdot,t)$ is the support function of a convex hypersurface $M_t$. Applying \Cref{thm:main} to $M^\e_t$, we obtain the desired estimate \eqref{ineq:sup-la} for the hypersurfaces $M_t$. Finally, it can be shown by the comparison principle that the hypersurfaces $M_t$ is a viscosity solution to \eqref{eq:flow}.
\end{proof}

\begin{lemma}\label{lem:sss}
If $p\in (-\infty,-n+2]$ and $f$ is a positive smooth function on $\bS^n$, then any generalized solution $\Si$ to \eqref{eq:MA} is a self-similar solution to the flow \eqref{eq:flow} with $\al=\fr{1}{1-p}$ and $f^{\al}$ instead of $f$.
\end{lemma}

\begin{proof}
Consider the function $a(t)= (1-(n\al+1)t)^{\fr{1}{n\al+1}}$, and let $M_0=\Si$ be a generalized solution to the $L_p$ Minkowski problem with $p\le -n+2$. We define $M_t = a(t)M_0$. Since the viscosity solution to \eqref{eq:flow} is unique by \Cref{thm:main-visc}, we only need to prove that $M_t$ is a viscosity solution to the flow \eqref{eq:flow}.

\bigskip

Let $M_0^-$ be a strictly convex, closed, smooth hypersurface which is enclosed by $M_0$. We may assume $M_0^-\cap M_0=\emptyset$. Otherwise, we consider $(1-\e)M_0^-$ and then take $\e\ra0$. 
Assume that the evolution $M_t^-$ of $M_0^-$ under the flow \eqref{eq:flow} touches $M_t$ for the first time $t=t_0\in (0,1/(n\al+1))$ at $x_0\in M_{t_0}^-\cap M_{t_0}$. Let $\nu_0$ be the normal of $M_{t_0}^-$ at $x_0$. If $u^-(\cdot,t)$ and $u(\cdot,t)$ are the support functions of $M_t^-$ and $M_t$, respectively, then $u^-(\nu_0,t_0)=u(\nu_0,t_0)\ge0$ since the origin is enclosed by $M_t$ for all $0\le t\le 1/(n\al+1)$. If $u^-(\nu_0,t_0)=u(\nu_0,t_0)>0$, then $u> u^->0$ in a neighborhood of $\nu_0$ and $0\le t<t_0$. Then by the standard regularity theory of Caffarelli \cite{Caf90b}, $u$ and $u^-$ satisfy \eqref{eq:MA} locally in a classical sense, which contradicts to $u^-(\nu_0,t_0)=u(\nu_0,t_0)$ by the comparison principle.

If $u^-(\nu_0,t_0)=u(\nu_0,t_0)=0$, then $x_0\in M_t$ for all $t\in[0,t_0]$ by definition of $M_t$. However, $M_{t_0}^-$ is strictly enclosed by $M_{t}^-$ for any $t_0>t\ge0$, so $x_0$ lies in the interior of $M_t^-$ for any $t\in [0,t_0)$, implying $M_0\cap M_0^-\not=\emptyset$, which is a contradiction. This proves that $M_t$ satisfies the first condition in \Cref{def:visc-sol}.

\bigskip

To prove that $M_t$ satisfies the second condition in \Cref{def:visc-sol}, let $M_0^+$ be a strictly convex, closed, smooth hypersurface which encloses $M_0$. As before, we may assume $M_0^+\cap M_0=\emptyset$. Otherwise, we consider $(1+\e)M_0^+$ and then take $\e\ra0$. 

Suppose that the evolution $M_t^+$ of $M_0^+$ under the flow \eqref{eq:flow} touches $M_t$ for the first time $t=t_0\in(0,1/(n\al+1))$ at $x_0\in M_{t_0}^+\cap M_{t_0}$. Let $\nu_0$ be the normal of $M_{t_0}^+$ at $x_0$. If $u^+(\cdot,t)$ is the support function of $M_t^+$, then $u^+(\nu_0,t_0)=u(\nu_0,t_0)\ge 0$. If $u^+(\nu_0,t_0)=u(\nu_0,t_0)>0$, then the standard regularity theory with the comparison principle gives a contradiction as above. 

Suppose that $u^+(\nu_0,t_0)=u(\nu_0,t_0)=0$. If $x_0$ is not the origin, then the line segment $l$ connecting $x_0$ and the origin is contained in $M_{t_0}$. Consequently, $l$ is enclosed by $M_{t_0}^+$. However, as $M_{t_0}^+$ is smooth, the normal vector $\nu_0$ must be orthogonal to $l$. Thus, $l$ is also contained in $M_{t_0}^+$, which contradicts the strict convexity of $M_{t_0}^+$. Therefore, $x_0$ must be the origin.

We can rotate the hypersurface $M_{t_0}$ if necessary, so that $\nu_0=(0,\ldots,0,-1)$. This implies that $M_{t_0}$ lies in the half-space $\{x\in \R^{n+1}:x_{n+1}\ge0\}$. Next, we consider a strip $S_\e=\{(x',x_{n+1})\in \R^{n+1}:0\le x_{n+1}\le \e\}$ for each small $\e>0$. Within the strip $S_\e$, both hypersurfaces $M_{t_0}$ and $M_{t_0}^+$ are represented by graphs of functions $U$ and $U^+$ respectively over subsets of $\R^n\times \{0\}$. Since $M_{t_0}^+$ encloses $M_{t_0}$, we have $0\le U^+\le U$, and the equality only holds at $x'=0$. As $M_{t_0}^+$ is a strictly convex and smooth hypersurface, we can find a constant $k$ such that $U^+\ge \frac{k}{2}|x'|^2$.

Let $G$ be the image of normals $M_{t_0}\cap S_\e$. By convexity, $G$ contains the ball of radius $\sqrt{2k\e}$ since $U\ge \fr{k}{2}|x'|^2$.
Then we have
\begin{align}\label{ineq:int_Gf}
\int_G f \ge C_n^{-1}f_{\min}(k\e)^{n/2}
\end{align}
for some dimensional constant $C_n>0$. 

Observe that $p\le 2-n\le 0$ and that $u\le \sqrt{2\e k^{-1}}$ on $G$. Then
\begin{align*}
\int_G u^{1-p}\,\mathrm{d}S({M_{t_0}},\cdot)
\le
(2\e k^{-1})^{-\fr{p}{2}}\int_G u\,\mathrm{d}S({M_{t_0}},\cdot)
=
(n+1)(2\e k^{-1})^{-\fr{p}{2}} V(M_{t_0}\cap S).
\end{align*}
Since $V(M_{t_0}\cap S)\le C_n\e (2\e k^{-1})^{\fr{n}{2}}$, we obtain
\begin{align*}
\int_G u^{1-p}\,\mathrm{d}S({M_{t_0}},\cdot) \le C(n,k,p)\e^{\fr{n-p+2}{2}}.
\end{align*}
This, together with \eqref{ineq:int_Gf} and the fact that $M_t$ is a generalized solution to \eqref{eq:MA}, implies 
\begin{align*}
\e^{\fr{2-p}{2}}\ge C(n,k,p)f_{\min},
\end{align*}
which is a contradiction if we take $\e\ra0$. This proves that $M_t$ satisfies the second condition in \Cref{def:visc-sol}.
\end{proof}

We next show that the inradius of a generalized solution \eqref{eq:MA} is controlled by its diameter.  

\begin{lemma}\label{lem:rho-}
For $p\le -n+2$, if a convex body $\Om$ is a generalized solution to \eqref{eq:MA}, then there is a constant dependent only on $n$, $p$, and $diam(\Om)$ such that $\rho_-(\Om)\geq c$.
\end{lemma}

\begin{proof}
Let $E$ be the John's ellipsoid associated to $\Om$ so that $E\su \Om\su (n+1)E$. This is the ellipsoid contained in $\Om$ with maximal volume. If the principal radii of $E$ are denoted by $0<r_1\le \cdots\le r_{n+1}$, then $r_1 \le \rho_-(E)\le \rho_-(\Om)$. With the use of $r_{n+1}\le u_{\max}$, it follows that $V(E)\le 2^{n+1}r_1\cdots r_{n+1}\le 2^{n+1}r_1u_{\max}^n$. Since $V(E)\ge C_nV(\Om)$, we can infer that
\begin{align}\label{eq:holderVS}
\rho_-(\Om) \ge  r_1\ge  C_n V(\Om)u_{\max}^{-n}.
\end{align}

To complete the proof, we aim to show that the volume of $\Om$ has a positive lower bound dependent only on the diameter of $\Om$. Note that $p\le -n+2\le 0$. 
Since $\Om$ is a generalized solution to \eqref{eq:MA}, we can see that
\begin{align*}
\int_{\bS^n} f =\int_{\bS^n} u^{1-p}\,\mathrm{d}S(\Om,\cdot)\le u_{\max}^{-p}\int_{\bS^n} u \,\mathrm{d}S(\Om,\cdot)=(n+1)u_{\max}^{-p}V(\Om)
\end{align*}
which gives a positive lower bound on the volume of $\Om$ in terms of diameter. Combining this with \eqref{eq:holderVS}, we obtain the desired result.
\end{proof}

Finally we prove the main result of the paper. 

\begin{proof}[Proof of \Cref{thm:main-MP}]
As $\Si$ is a generalized solution to \eqref{eq:MA}, it follows from \Cref{lem:sss} that $\Si$ is a self-similar solution to the flow \eqref{eq:flow} with $\al=\fr{1}{1-p}$ and $f^\al$ instead of $f$. Specifically, $M_t=a(t)\Si$ where $a(t)=(1-(n\al+1)t)^{\fr{1}{n\al+1}}$ is the solution to the flow \eqref{eq:flow} starting from $\Si$ with maximal existence time $T_{\max}=1/(n\al+1)$. By using \Cref{lem:rho-}, we have that $\rho_-(\Sigma)$ is controlled in terms of $n$, $\al$, and $\operatorname{diam}(\Si)$. Therefore, the dependence $\rho_-(M_0)$ on the constant $C$ can be absorbed. Hence, applying \Cref{thm:main-visc} with some small $T<T_{\max}$ only depending on $\al,n,\max_{\mathbb{S}^n}f,\operatorname{diam}(\Si)$, we obtain
\begin{align*}
\la_i(\Si)=a(T)\la_i(M_T)\le a(T)C(1+T^{-1}),
\end{align*}
where the constant $C(\al,n,\|f\|_{C^2(\mathbb{S}^n)},\min_{\mathbb{S}^n}f,\operatorname{diam}(\Si))$ is given in \Cref{thm:main-visc}. This completes the proof.
\end{proof}

\section{Proof of \Cref{thm:counter-MP}} \label{sec:thmpf2}

We finally prove \Cref{thm:counter-MP} in this section. The proofs of \Cref{thm:counter-MP} (i) and (ii) are provided in \Cref{sec:contraction} and \Cref{sec:example}, respectively.

\subsection{Contraction mapping} \label{sec:contraction}

In this section, we provide the proof of \Cref{thm:counter-MP} (i). Indeed, we prove a slightly more general statement:

\begin{theorem} \label{thm:counter-MP2}
If $-n+1 < p < 1$, then there exists a generalized solution $\Sigma$ to \eqref{eq:MA} such that $\Sigma$ is a hypersurface of $($at most\,$)$ class $C^{k, \gamma}$ with $\gamma\in (0,1]$ for $k+\gamma=\frac{n+p}{n+p-1}$, and $f$ is a positive smooth function.
\end{theorem}

Suppose that a \textit{generalized solution} to the $L_p$ Minkowski problem includes the graph of a convex function $v:U (\subset \mathbb{R}^n)\to \mathbb{R}$. Suppose  $v \in C^2(V)$ for some $V\subset U$. Then,
\begin{equation*}
    \tfrac{\det D^2v}{(1+|Dv|^2)^{\frac{n+2}{2}}}=f^{-1}( \nu)\Big( \tfrac{x\cdot Dv-v}{\sqrt{1+|Dv|^2}}\Big)^{1-p}
\end{equation*}
holds in $V$, where $\nu=\frac{(Dv,-1)}{\sqrt{1+|Dv|^2}}$. We assume that $f=1$ on the graph $\{(x,v(x)):x\in U\}$, and $v:B_{1+R}(0)\setminus B_{1}(0)\to \mathbb{R}$ is a radial function. Then, the one-variable function $\bar v(r)$, defined by 
\begin{equation*}
v(x)=\bar v(|x|-1),    
\end{equation*}
satisfies
\begin{equation}\label{eq:radial.solution}
    \bar v_{rr}\bar v_r^{n-1}=(\bar v_r+r\bar v_r-\bar v)^{1-p}(1+r)^{n-1}(1+\bar v_r^2)^{\frac{n+p+1}{2}}
\end{equation}
for $r\in (0,R)$. Note that if $|r|\ll 1$ and $|\bar v|\ll |\bar v_r|\ll 1$ then $\bar v_{rr}\bar v_r^{n-1}\approx \bar v_r^{1-p}$. Hence, it is useful to consider the solution
\begin{equation*}
    h(r):= \tfrac{m}{1+m}m^{\frac{1}{m}}r^{\frac{1+m}{m}}, \qquad \text{where} \;\; m:=n+p-1,
\end{equation*}
to the model equation
\begin{equation}\label{eq:sol.model.eq}
    h_{rr}=h_r^{1-m}.
\end{equation}

\bigskip

In this section, we will show that given $p,n$ with $m>0$, there are $r_0>0$ and $\bar v:I_0\to \mathbb{R}$, where $I_0:=(0,r_0]$, such that $\bar v(0)=\bar v_r(0)=0$, $\bar v \in C^2(I_0)$, and
\begin{align*}
    & \bar v=(1+O(r^{\delta}))h, && \bar v_r=(1+O(r^{\delta}))h_r, && \bar v_{rr}=(1+O(r^{\delta})) h_{rr},
\end{align*}
where $\delta:=\min\{1,\frac{2}{m}\}$. This will lead us to the proof of \Cref{thm:counter-MP} (i). 

 \bigskip

Given $\varphi\in C^k{(I)}$ with an interval $I$, we define a weighted norm on $C^k_w(I)$ by
\begin{equation*}
    \|\varphi\|_{C^k_w(I)}=\max_{l=0,\cdots,k}\|r^{l-1}h_r^{-1}\varphi^{(l)}\|_{L^\infty(I)},
\end{equation*}
where $\varphi^{(l)}$ is the $l$-th order derivative of $\varphi$. In addition, we define some linear operators
\begin{align*}
&   [\varphi]_r:=\frac{\varphi_r}{h_r}, && [\varphi]_s:=[\varphi]_r-\frac{\varphi}{rh_r}, && [\varphi]_{rr}:=\frac{\varphi_{rr}}{h_{rr}}. \end{align*}

\bigskip

Let $w=\bar v-h$, and manipulate $\bar v_{rr}\bar v_r^{n-1}-h_{rr}h_r^{n-1}$  by using \eqref{eq:radial.solution} and \eqref{eq:sol.model.eq} to derive
\begin{align*}
    h_r^{n-1}w_{rr}+(n-1)h_r^{n-2}h_{rr}w_r=(1-p)h_r^{-p}w_r+ h_r^{1-p}E[w],
\end{align*}
where
\begin{align*}
 E[w]:=P_1[w]+P_2[w]+Q[w]+R_1[w]-R_2[w]-R_3[w],
\end{align*}
with
\begin{align*}
    &P_1[w]:=((1+r)^{n-1}-1) (1+h_r^2(1+[w]_r)^2)^{\frac{m+2}{2}}(1+\tfrac{1}{m+1}r+[w]_r+r[w]_s)^{1-p},\\
    &P_2[w]:=((1+h_r^2(1+[w]_r)^2)^{\frac{m+2}{2}}-1)(1+\tfrac{1}{m+1}r+[w]_r+r[w]_s)^{1-p},\\
    &Q[w]:=(1+\tfrac{1}{m+1}r+[w]_r+r[w]_s)^{1-p}-(1+[w]_r)^{1-p},
\end{align*}
and
\begin{align*}
    &R_1[w]:=(1+[w]_r)^{1-p}-1-(1-p)[w]_r,\\
    &R_2[w]:=[w]_{rr}((1+[w]_r)^{n-1}-1),\\
    &R_3[w]:=(1+[w]_r)^{n-1}-1-(n-1)[w]_r.
\end{align*}
Note that we utilized \eqref{eq:sol.model.eq} for formulating $R_2$ and $R_3$. Also, $Q\equiv R_1 \equiv 0$ holds in the case $p=1$. Next, by using \eqref{eq:sol.model.eq} and $h_r=(mr)^{\frac{1}{m}}$, we can simplify the equation of $w$ as
\begin{equation}\label{eq:diff.model}
Lw:=(r^{1-\frac{1}{m}} w_r)_r=m^{\frac{1}{m}-1}E[w].
\end{equation}

Let us provide the main result in this section.
\begin{theorem}\label{thm:exist.cont.map}
Given $p,n$ with $m>0$, there are $C_0,r_0>0$ and $w\in C^2(I_0)$ with $I_0:=(0,r_0]$ such that $w$ is a solution to \eqref{eq:diff.model} on $I_0$ satisfying 
\begin{equation}\label{eq:diff.bound}
    \|w\|_{C^2_w((0,r])}\leq C_0r^{\delta} \leq \tfrac{1}{10} \min\{m,m^{-1}\}
\end{equation}
for every $r\leq r_0$, where $\delta:=\min\{1,\frac{2}{m}\}$.
\end{theorem}
By using this theorem, we can prove \Cref{thm:counter-MP} (i), which follows from \Cref{thm:counter-MP2}.

\begin{proof}[Proof of \Cref{thm:counter-MP2}]
We recall $w$ in \Cref{thm:exist.cont.map} so that we define $\bar v=h+w$ and then define $v:B_{1+r_0}(0)\to \mathbb{R}$ by
\begin{enumerate}
    \item $v(x)\equiv 0$ for $|x|\leq 1$,
    \item $v(x)=\bar v(|x|-1)$ for $|x| \in (1,r_0+1)$.
\end{enumerate}
Then, \eqref{eq:diff.bound} implies $\frac{|w|}{h},\frac{|w_r|}{h_r},\frac{|w_{rr}|}{h_{rr}}\leq \frac{1}{2}$, and therefore $v$ is a positive and strictly convex function on $B_{1+r_0}(0)\setminus B_1(0)$. Also, $v(x)$ is smooth in $\{1<|x|<1+r_0\}$ by the standard regularity theorems for the Monge--Amp\`ere type equations. Thus, $v$ is of class $C^{k,\gamma}$ with $\gamma\in (0,1]$ for $k+\gamma=1+\frac{1}{m}$. Notice that the assumption $p<1$ has not been used so far.

\bigskip

We next choose a rotationally symmetric convex body $\Omega\subset \mathbb{R}^n\times [0,+\infty)$ such that 
\begin{enumerate}
    \item $\partial \Omega \cap \{x_{n+1}\leq \frac{1}{2} \bar v(r_0)\}= \{(x,v(x)): v(x)\leq \frac{1}{2}  \bar v(r_0)\}$,
    \item $\partial \Omega$ is smooth, and its Gauss curvature is positive in $\{x_{n+1}>0\}$. 
\end{enumerate}
To prove that $\Omega$ is a generalized solution of \eqref{eq:MA}, we need to show \eqref{eq:gen-sol}. Note that we have
\begin{equation*}
\int_{E \cap \lbrace v>0 \rbrace} u^{1-p} \,\mathrm{d}S(\Omega, z) = \int_{E \cap \lbrace v>0 \rbrace} f \,\mathrm{d}\sigma \quad\text{for all Borel sets } E \subset \mathbb{S}^{n}
\end{equation*}
for some positive function $f \in C^{\infty}(\mathbb{S}^{n})$, where $u$ is the support function of $\Omega$. Since $u=0$ on $\lbrace v=0 \rbrace$ and $p<1$, we obtain
\begin{equation} \label{eq:pasting1}
\int_{E \cap \lbrace v=0 \rbrace} u^{1-p} \,\mathrm{d}S(\Omega, z) = 0.
\end{equation}
Moreover, since $\lbrace v=0 \rbrace \subset \lbrace -e_{n} \rbrace$ we have
\begin{equation} \label{eq:pasting2}
\int_{E \cap \lbrace v=0 \rbrace} f \,\mathrm{d}\sigma = 0.
\end{equation}
Therefore, it follows from \eqref{eq:pasting1} and \eqref{eq:pasting2} that
\begin{equation*}
\int_{E} u^{1-p} \,\mathrm{d}S(\Omega, z) = \int_{E} f \,\mathrm{d}\sigma \quad\text{for all Borel sets } E \subset \mathbb{S}^{n},
\end{equation*}
which proves \eqref{eq:MA}. We finish the proof by observing that $\partial \Omega$ is  a hypersurface of at most class $C^{k,\gamma}$ with $\gamma\in (0,1]$ for $k+\gamma=1+\frac{1}{m}$ as the regularity of $v$.
\end{proof}

\bigskip

To prove \Cref{thm:exist.cont.map}, we employ the contraction mapping. We define $w_0\equiv 0$ and 
\begin{equation*}
    w_1(r):=\int_0^r s^{\frac{1}{m}-1}\int _0^s m^{\frac{1}{m}-1} E(w_0)(t) \,\mathrm{d}t \,\mathrm{d}s,
\end{equation*}
which solves $Lw_1=m^{\frac{1}{m}-1} E(w_0)$. Then, for $i\in \mathbb{N}$,  we inductively define
\begin{equation*}
    w_{i+1}(r):=w_i(r)+\int_0^r s^{\frac{1}{m}-1} \int _0^s m^{\frac{1}{m}-1}( E(w_i)(t)- E(w_{i-1})(t)) \,\mathrm{d}t \,\mathrm{d}s
\end{equation*}
so that $Lw_{i+1}=m^{\frac{1}{m}-1} E(w_i)$. Thus, it is important to estimate the error term $E$.

\bigskip

\begin{lemma}\label{lem:induct.error.zero}
There is some constant $C_1$ depending only on $n,p$ with $m>0$ such that
\begin{equation*}
        |E(w_0)(r)| \leq C_1r^\delta
\end{equation*}
holds for every $r\leq 1$, where $\delta=\min\{1,\frac{2}{m}\}$.
\end{lemma}

\begin{proof}
Observe $R_1[0]=R_2[0]=R_3[0]=0$, $Q[0]=(1+\frac{1}{m+1}r)^{1-p}-1$, and
    \begin{align*}
    &P_1[0]=((1+r)^{n-1}-1) (1+h_r^2)^{\frac{m+2}{2}}(1+\tfrac{1}{m+1}r)^{1-p},\\
    &P_2[0]=((1+h_r^2)^{\frac{m+2}{2}}-1)(1+\tfrac{1}{m+1}r)^{1-p}.
\end{align*}
Since $w_0:=0$, we can obtain the desired result by the Taylor remainder theorem.
\end{proof}

\bigskip

\begin{proposition}\label{prop:Jensen}
Given $a,b>0$ and $q\in \mathbb{R}$, the following holds
\begin{equation*}
    |a^q-b^q|\leq |q| |a-b| (a^{q-1}+b^{q-1}).
\end{equation*}
\end{proposition}

\begin{proof}
Let $a>b$ and $q>1$. Then, the convexity of $f(x)=x^q$ implies
\begin{equation*}
    \tfrac{f(a)-f(b)}{a-b}\leq f'(a), 
\end{equation*}
which is the desired result. In the same manner, we can complete the proof by using the convexity of $f(x)=x^q$ for $q<0$ and the concavity of $f(x)=x^q$ for $0<q<1$. The cases $q=0$ and $q=1$ are trivial.
\end{proof}

\bigskip

\begin{lemma}\label{lem:induct.error.general}
There is some constant $C_2$ depending only on $n,p$ with $m>0$ such that if some $\varphi,\psi \in C^2(I)$, where $I=(0,r]$ for some $r\leq 1$, 
 satisfy $\|\varphi\|_{C^2_w(I)},\|\psi\|_{C^2_w(I)}\leq \frac{1}{4}$, then the following hold: for $j=1,2$
\begin{align}\label{eq:quad.error.expo}
    |P_j[\varphi](r)-P_j[\psi](r)|\leq C_2r^\delta\|\varphi-\psi\|_{C^2_w(I)},
\end{align}
and
\begin{align}\label{eq:quad.error.comb}
    |Q[\varphi](r)-Q[\psi](r)|\leq C_2(r^\delta+\|\varphi\|_{C^2_w(I)}+\|\psi\|_{C^2_w(I)})\|\varphi-\psi\|_{C^1_w(I)},
\end{align}
and for $j=1,2,3$
\begin{align}\label{eq:quad.error.poly}
     |R_j[\varphi](r)-R_j[\psi](r)|\leq C_2(\|\varphi\|_{C^2_w(I)}+\|\psi\|_{C^2_w(I)})\|\varphi-\psi\|_{C^2_w(I)},
\end{align}
where $\delta=\min\{1,\frac{2}{m}\}$.
\end{lemma}

\bigskip

\begin{proof}
For the purpose of brevity, we define
\begin{align}\label{eq:error.module.def} 
 & X_w=1+h_r^2(1+[w]_r)^2, && Y_w=1+\tfrac{1}{m+1}r+[w]_r+r[w]_s, && Z_w=1+[w]_r.
\end{align}
Note that if $\|w\|_{C^2_w(I)}\leq \frac{1}{4}$ we have 
\begin{equation*}
    Y_w, Z_w \geq \tfrac{3}{4}.
\end{equation*}
We first show \eqref{eq:quad.error.expo} for $j=1$. Since $(1+r)^{n-1}\leq 1+Cr$ for some $C(n)$ in $\{0<r\leq 1\}$, remembering $1\leq X_w \leq C$ and $\frac{3}{4} \leq Y_w \leq C$ we apply \Cref{prop:Jensen} so that
\begin{align*}
|P_1[\varphi]-P_1[\psi]|&\leq Cr\left(\Big|X_{\varphi}^{\frac{m+2}{2}}-X_{\psi}^{\frac{m+2}{2}}\Big|Y_{\varphi}^{1-p}+X_{\psi}^{\frac{m+2}{2}}\Big|Y_{\varphi}^{1-p}-Y_{\psi}^{1-p}\Big|\right)\\
& \leq Cr\left(|X_{\varphi}-X_{\psi}|+|Y_{\varphi}-Y_{\psi}|\right)
\end{align*}
 for some $C(n,p)$. Thus, by observing
\begin{equation}\label{eq:error.module.XY}
    \begin{split}
&|X_{\varphi}-X_{\psi}|\leq Ch_r^2|[\varphi-\psi]_r|\leq Cr^{\frac{2}{m}}\|\varphi-\psi\|_{C^1_w(I)},\\
&|Y_{\varphi}-Y_{\psi}|\leq C|[\varphi-\psi]_r+r[\varphi-\psi]_s|\leq C\|\varphi-\psi\|_{C^1_w(I)}, 
    \end{split}
\end{equation}
we can obtain \eqref{eq:quad.error.expo} for $j=1$. Similarly, we can derive
\begin{align*}
    |P_2[\varphi]-P_2[\psi]|&\leq   \Big|X_{\varphi}^{\frac{m+2}{2}}-X_{\psi}^{\frac{m+2}{2}}\Big|Y_{\varphi}^{1-p}+(X_{\psi}^{\frac{m+2}{2}}-1)\Big|Y_{\varphi}^{1-p}-Y_{\psi}^{1-p}\Big|\\
   & \leq C|X_{\varphi}-X_{\psi}|+C(X_{\psi}-1)|Y_{\varphi}-Y_{\psi}|.
\end{align*}
Thus, combining with $X_{\psi}-1=h_r^2Z_{\psi}^2\leq Cr^{\frac{2}{m}}$ and \eqref{eq:error.module.XY} yields \eqref{eq:quad.error.expo} for $j=2$. In the same manner, by using \Cref{prop:Jensen} with $Z_w\geq \frac{3}{4}$ and
    \begin{align*}
    |R_2[\varphi]-R_2[\psi]|\leq   |[\varphi-\psi]_{rr}| |Z_{\varphi}^{n-1}-1|+|[\psi]_{rr}||Z_{\varphi}^{n-1}-Z_{\psi}^{n-1}|,
\end{align*}
we can show \eqref{eq:quad.error.poly} for $j=2$.

\bigskip

Next, to show \eqref{eq:quad.error.poly} for $j=1$, we consider
\begin{equation*}
    g(t)=(1+[\varphi]_r+t)^{1-p}-1-(1-p)([\varphi]_r+t),
\end{equation*}
and apply the Taylor's theorem so that we can obtain $C(p)$ satisfying
\begin{equation*}
    |g(t)-g(0)-(1-p)((1+[\varphi]_r)^{-p}-1)t|\leq C|t|^2
\end{equation*}
for $|t|\leq \frac{1}{2}$. Hence, putting $t=[\psi-\varphi]_r$, we get \eqref{eq:quad.error.poly} for $j=1$. Note that the proof for \eqref{eq:quad.error.poly} with $j=3$ is identical. To prove the last estimate \eqref{eq:quad.error.comb}, we consider
\begin{equation*}
    g(t)=(\hat Y_{\varphi,\psi}+t)^{1-p}-(1+[\varphi]_r+t)^{1-p},
\end{equation*}
 where
 \begin{equation*}
     \hat Y_{\varphi,\psi}:=1+\tfrac{1}{m+1}r+[\varphi]_r+r[\psi]_s.
 \end{equation*}
 Then, the Taylor remainder theorem says that there is $C(p)$ such that
 \begin{equation*}
     |g(t)-g(0)-(1-p) [\hat Y_{\varphi,\psi}^{-p}-(1+[\varphi]_r)^{-p}] \, t\;|\leq C|t|^2
 \end{equation*}
holds for $|t|\leq \frac{1}{2}$. Observing $\hat Y_{\varphi,\psi},Z_{\varphi} \geq \frac{3}{4}$, we apply \Cref{prop:Jensen} to get
\begin{equation*}
    |\hat Y_{\varphi,\psi}^{-p}-(1+[\varphi]_r)^{-p}| \leq Cr.
\end{equation*}
Therefore, putting $t=[\psi-\varphi]_r$ yields
\begin{equation*}
     |Q[\psi]-g(0)|\leq C(r+ \|\varphi-\psi\|_{C^1_w(I)})\|\varphi-\psi\|_{C^1_w(I)}.
 \end{equation*}
 On the other hand, using \Cref{prop:Jensen} and $\hat Y_{\varphi,\psi},Y_{\varphi} \geq \frac{3}{4}$, we can obtain
 \begin{equation*}
     |g(0)-Q[\varphi]|=|\hat Y_{\varphi,\psi}^{1-p}-Y_{\varphi}^{1-p}|\leq Cr|[\psi-\varphi]_s|\leq Cr\|\psi-\varphi\|_{C^1_w(I)}.
 \end{equation*}
 This completes the proof.
\end{proof}

\bigskip

\begin{lemma}\label{lem:induct.integral}
There is some constant $C_3$ depending only on $n,p$ with $m>0$ such that if $|g(r)|\leq r^\delta$ then
\begin{equation*}
    \hat w(r):=\int_0^r s^{\frac{1}{m}-1}\int _0^s m^{\frac{1}{m}-1} g(t) \,\mathrm{d}t \,\mathrm{d}s
\end{equation*}
satisfies $\|\hat w \|_{C^2_w((0,r])}\leq C_3 r^\delta$ for $r>0$.
\end{lemma}

\begin{proof}
Since
\begin{equation*}
    m^{1-\frac{1}{m}}\hat w_{rr}=r^{\frac{1}{m}-1}  g(r)  +(\tfrac{1}{m}-1) r^{\frac{1}{m}-2}\int _0^r  g(t) \,\mathrm{d}t,
\end{equation*}
we have
\begin{equation*}
    m^{1-\frac{1}{m}}|\hat w_{rr}|\leq r^{\frac{1}{m}-1}|g(r)|  +|\tfrac{1}{m}-1| r^{\frac{1}{m}-2}\int _0^r|g(t)| \,\mathrm{d}t \leq Cr^{\frac{1}{m}-1+\delta}.
\end{equation*}
Similarly, we can compute $|\hat w|$ and $|\hat w_r|$ to obtain the desired result.
\end{proof}

\bigskip

\begin{proof}[Proof of \Cref{thm:exist.cont.map}]
We recall the constants $C_1,C_2,C_3,\delta$ in the above lemmas and we choose $r_0 \in (0,1]$ satisfying
\begin{align}\label{eq:coeff.contaction.range}
 8C_1C_3r_0^{\delta} +  8C_2C_3(1+4C_1C_3)r_0^{\delta}\leq 1.
\end{align}
Then, combining with \Cref{lem:induct.error.zero} and \Cref{lem:induct.integral} yields
\begin{equation}\label{eq:induct.base.step}
    2\|w_1\|_{C^2_w(I)}\leq 2C_1C_3r^{\delta}\leq \tfrac{1}{4}
\end{equation}
for every $r\leq r_0$, where $I=(0,r]$.

\bigskip
Then, we will inductively show that for each $i\in\mathbb{N}$ we have
\begin{equation}\label{eq:fixed.point}
\|w_i-w_{i-1}\|_{C^2_w(I)}\leq 2^{-(i-1)}\|w_1\|_{C^2_w(I)},
\end{equation}
and 
\begin{align}\label{eq:induct.est}
    \|w_i\|_{C^2_w(I)}\leq (2-2^{-(i-1)})\|w_1\|_{C^2_w(I)}\leq \tfrac{1}{4}.
\end{align}
First of all, the base case $i=1$ is obvious by $w_0=0$ and \eqref{eq:induct.base.step}.

\bigskip

Next, we assume that the above inequalities hold for $i\leq l$. Then, remembering \eqref{eq:induct.est}, we can apply \Cref{lem:induct.error.general} to obtain
\begin{equation*}
    |E[w_l]-E[w_{l-1}]|\leq 4C_2(r^\delta+4\|w_1\|_{C^2_w(I)} )\|w_l-w_{l-1}\|_{C^2_w(I)}.
\end{equation*}
Therefore,  \eqref{eq:induct.base.step} implies
\begin{equation*} 
    |E[w_l]-E[w_{l-1}]|\leq  4C_2(1+4C_1C_3)\, r^\delta\, \|w_{l}-w_{l-1}\|_{C^2_w(I)}.
\end{equation*}
Hence, by \Cref{lem:induct.integral} and \eqref{eq:coeff.contaction.range}, we have
\begin{equation*} 
   \|w_{l+1}-w_{l}\|_{C^2_w(I)}\leq \tfrac{1}{2}\|w_{l}-w_{l-1}\|_{C^2_w(I)}.
\end{equation*}
Thus, combining  \eqref{eq:fixed.point} with $i=l$ shows \eqref{eq:fixed.point} for $i=l+1$, and therefore we have \eqref{eq:induct.est} for $i=l+1$. 

\bigskip

To conclude, we add \eqref{eq:fixed.point} for all $i\geq 1$ to get pointwise limits
\begin{align*}
&\bar w^{[0]}(r):=\lim_{i\to +\infty}w_i(r), && \bar w^{[1]}(r) :=\lim_{i\to +\infty}\tfrac{d}{dr} w_i(r), && \bar w^{[2]}(r) :=\lim_{i\to +\infty}\tfrac{d^2}{dr^2} w_i(r).
\end{align*}
for $r\in (0,r_0]$. Moreover, for each $k=1,2$  
\begin{equation*}
    \bar w^{[k-1]}(r_1)-\bar w^{[k-1]}(r_2)=\int^{r_1}_{r_2} \bar w^{[k]}(r)dr
\end{equation*}
holds whenever $0<r_1<r_2\leq r_0$. Hence, $\bar w^{[0]}$ and $\bar w^{[1]}$ are continuous, and thus the fundamental theorem of calculus (FTC) says $\frac{d}{dr}\bar w^{[0]}=\bar w^{[1]}$. In addition, the equation \eqref{eq:radial.solution} says
\begin{equation*}
h_r^{m-1}Z_{\bar w^{[0]}}^{n-1}(h_{rr}+\bar w^{[2]})=(1+r)^{n-1}X_{\bar w^{[0]}}^{\frac{m+2}{2}}Y_{\bar w^{[0]}}^{1-p}  
\end{equation*}
where $X,Y,Z$ are given in \eqref{eq:error.module.def}. Hence, $\bar w^{[2]}$ is also continuous, and therefore we have $\frac{d}{dr}\bar w^{[1]}=\bar w^{[2]}$ by the FTC again. Thus, $\bar w^{[0]} \in C^2$ is a solution to \eqref{eq:diff.model} satisfying
\begin{equation*}
     \|\bar w^{[0]}\|_{C^{2}_w((0,r])}\leq 2\|w_1\|_{C^2_w((0,r])} \leq 2C_1C_3r^\delta=:C_0r^\delta
\end{equation*}
for every $r\in (0, r_0]$ as desired. Note that we may choose smaller $r_0$ to have 
\begin{equation*}
C_0r_0^{\delta}\leq \tfrac{1}{10} \min\{m,m^{-1}\}.
\end{equation*}
This completes the proof.
\end{proof}

\subsection{Example revisited} \label{sec:example}

In this section, we prove \Cref{thm:counter-MP} (ii) by revisiting the example provided by Chou--Wang \cite{CW06} and Bianchi--B\"or\"oczky--Colesanti \cite{BBC20}. While the example constructed in the previous section has a flat side, the example below does not.

\begin{proof} [Proof of \Cref{thm:counter-MP}]
For each $p \in (-n+1,1) \cup (1, n+1)$, Chou and Wang \cite[Section 6]{CW06} proved that there exists a generalized solution of \eqref{eq:MA} that is not smooth even for a smooth positive function $f$. Let us recall this example and prove that the associated hypersurface $\Sigma$ is not of $C^{1,1}$ when $p \in (1, n+1)$. Let $\tilde{u}$ be the restriction of $u$ on the tangent hyperplane of the $n$-sphere at the south pole. If $\tilde{u}$ is given by $\tilde{u}(y) = |y|^{2\alpha}$ with $\alpha = n/(n-p+1)$, then it satisfies
\begin{equation*}
\det D^{2}\tilde{u}= (2\alpha)^{n}(2\alpha-1) \tilde{u}^{p-1} \quad \text{in } \mathbb{R}^{n}
\end{equation*}
in the generalized sense. Thus, $u$ is a generalized solution of \eqref{eq:MA} near the south pole, with a smooth positive function $f$. Note that $f$ is given by
\begin{equation*}
f(z) = (2\alpha)^{n} (2\alpha-1) \left( \frac{1}{|z_{n+1}|} \right)^{n+2p}
\end{equation*}
near the south pole.

We define
\begin{equation*}
\bar{v}(r) = c r^{\frac{2\alpha}{2\alpha-1}} = c r^{\frac{2n}{n+p-1}}, \quad c=\frac{2\alpha-1}{(2\alpha)^{2\alpha/(2\alpha-1)}},
\end{equation*}
and $v(x) = \bar{v}(|x|)$, $x \in \mathbb{R}^{n}$. Then, $\Sigma$ is represented by $\lbrace (x, v(x)) \rbrace$ near the south pole. Indeed, we have for $y = Dv(x)$
\begin{equation*}
\tilde{u}(y) = |y|^{2\alpha} = (2\alpha)^{-\frac{2\alpha}{2\alpha-1}} |x|^{\frac{2\alpha}{2\alpha-1}} = |x| \bar{v}(|x|) - \bar{v}(|x|) = \langle (x, v(x)), (y, -1) \rangle.
\end{equation*}
We observe that $v$ is of $C^{1, \frac{n-p+1}{n+p-1}}$, not $C^{1,1}$, provided $p \in (1, n+1)$.
\end{proof}


\end{document}